\documentclass[12pt]{amsart}

\usepackage{amsmath,amsthm}

\newtheorem{theorem}{Theorem}[section]

\newtheorem{lemma}[theorem]{Lemma}

\newtheorem{corollary}[theorem]{Corollary}
\newtheorem{proposition}[theorem]{Proposition}
\newtheorem{remark}[theorem]{Remark}

\def\C{\mathbb C}
\def\R{\mathbb R}

\begin{document}

\title[Division subspaces and integrable kernels ]{Division subspaces and integrable kernels}

\author
{Alexander I. Bufetov}
\address
{Alexander I. BUFETOV: 
Aix-Marseille Universit\'e, Centrale Marseille, CNRS, Institut de Math\'ematiques de Marseille, UMR7373, 39 Rue F. Joliot Curie 13453, Marseille, France;
Steklov Institute of Mathematics, Moscow, Russia}
\email{bufetov@mi.ras.ru, alexander.bufetov@univ-amu.fr}

\author{Roman V. Romanov}
\address{Roman V. ROMANOV: Department of Physics, Saint-Petersburg State University, Saint-Petersburg, Russia; ITMO University, Saint-Petersburg, Russia}
\email{morovom@gmail.com}

\begin{abstract}
In this note we prove that the reproducing kernel of a Hilbert space satisfying the division property
has integrable form, is locally of trace class, and the Hilbert space itself is a Hilbert space of holomorphic functions. 
\end{abstract} 

\maketitle

\section{Introduction}

In applications one often encounters one-dimensional determinantal point processes governed by  
kernels having integrable form 
\begin{equation}\label{integ-rep}
K(x,y)=\displaystyle \frac{A(x)B(y) - B(x)A(y)}{x-y}.
\end{equation}
These are, for example, the sine-kernel \cite{Dyson}, Bessel kernel \cite{TW-Bessel}, Airy kernel \cite{TW-Airy}, gamma kernel \cite{BO-gamma, GO-Adv}, discrete sine-kernel \cite{BOO}, and discrete Bessel kernel \cite{BOO,johansson}. 
If an integrable kernel $K$ 
induces the operator of orthogonal projection onto a closed subspace $H\subset L^2(\R )$, then 
it is shown in \cite{Buf-gibbs} that 
the subspace $L$ has the following {\it division property}: for $p\in {\mathbb R}$, $f\in H$, if $f(p)=0$, then
\begin{equation}\label{weakdiv}
\displaystyle \frac{f(x)}{x-p}\in H.
\end{equation} 
The division property (\ref{weakdiv})  lies at the centre of the proof in \cite{Buf-gibbs} 
of the quasi-invariance of the corresponding  processes under the group of compactly 
supported diffeomorphisms of $\R$, an analogue of the Gibbs property, cf. \cite{sinai}, for determinantal point processes.  Our first result, Theorem \ref{integrable}, establishes that 
the reproducing kernel of a  real subspace with the division property (\ref{weakdiv}) admits an integrable representation (\ref{integ-rep}).

The division property  (\ref{weakdiv}) naturally arises in the theory of Hilbert spaces of entire functions called the de Branges spaces \cite{deBr}; determinantal point processes corresponding to de Branges spaces are studied in \cite{Buf-Shirai}.  A stronger division property  of a de Branges space $ \mathcal H $  requires, for $ f \in \mathcal H $,  $k\in\R$, that the function $ \frac{ f ( \cdot ) - f ( k ) }{ (\cdot ) - k } $ also belongs to the space $ \mathcal H $. The strong division property  characterizes the important class of regular de Branges spaces.  Recall the following characterization of de Branges spaces:  for any de Branges space $ \mathcal H $ there exists a Hermite--Biehler function $ E $ such that $ \mathcal H $ coincides with the set of entire functions $ f $ such that $ f/E $ and $ f^* / E $ belong to the Hardy class $ H^2 ( \C_+ ) $, and the identity $ \| f \|_{ \mathcal H } = \| f/ E \|_{ L^2 ( \R ) } $ holds for all $ f \in \mathcal H $. It is then natural to search for a description of the reproducing kernel Hilbert subspaces in $ L^2 ( \R , \mu ) $ subject to axiomatic conditions from \cite{Buf-gibbs}  that ensure quasi-invariance of the corresponding determinantal processes in terms of functional parameter(s) in the spirit of  the de Branges theorem just mentioned. In the case when $ \mu $ is the Lebesgue measure on $ \R $ such a characterization  immediately follows from the elementary theory of the shift operator, see the example in Section \ref{dca}.  

We next study the analytic properties of reproducing kernel Hilbert subspaces in $ L^2(\R , \mu )$ satisfying either of the division properties above. In Theorem \ref{strdiv-anal} we establish that if the strong division property  holds then all the functions in the space $H$ are meromorphic on a natural domain with all poles contained in a discrete set depending on the space $ H $ only.  
A standard example is the Paley--Wiener space, the range of the sine-kernel. 
The strong division property fails, however, in many important examples such as the Bessel and Airy kernels. In this situation one can at best expect  that the ratios of elements of a space with the division property be meromorphic on a natural domain in the sense of 
Theorem \ref{strdiv-anal}. In Theorem \ref{weakdiv-anal}  we establish exactly that. 
In Corollary \ref{loctraceclass} we show that the weak division property implies that the kernel is locally of  trace class. 

A brief outline of our argument is as follows.  The division properties can be seen as conditions for boundedness of the corresponding operators in $ H $. The functions from $ H $ can then be expressed as matrix elements of the resolvent of the division operators, and the desired analyticity follows from the analyticity of the resolvent.  The proof of the integrability of the kernel localizes the argument from the original derivation \cite[Theorem 23]{deBr} of the axiomatic characterization of  de Branges spaces; 
see  also \cite{Aleman} for related calculations in a different context. 

We fix some notation.
Throughout the paper $\mu$ is a sigma-finite measure on $ \R$ and $ U \subset \R $ is a Borel subset of $\R$ satisfying  $ \mu (\R\setminus U)=0$.  The symbol $ \mathrm{supp} \mu$ stands  for the smallest closed support of $\mu$.  A linear set $ H $ of Borel functions $ f $ on $ U $ satisfying $\int |f|^2 d\mu < \infty $ is called a reproducing kernel subspace in $ L^2 ( \R , \mu ) $ if the set is closed in the $ L^2 ( \R ,  \mu ) $ norm, and the functional $ f \mapsto f ( k ) $ is bounded in the $ L^2 ( \R , \mu )$ norm for all $ k \in U $. The reproducing kernel of the space $ H $ at a point $ w \in U $ is denoted either by $ K_w $ or $ K ( \cdot , w ) $, so $ f(w) = \langle f , K_w \rangle $. The essential spectrum, $ \sigma_{ ess } ( D ) $, of a closed operator $ D $ is the union of all non-isolated points in the spectrum 
of $ D $ and isolated points whose corresponding invariant subspaces have  infinite dimension.
 
{\bf{Acknowledgements.}} 
We are deeply grateful to Alexei Klimenko  and Yanqi Qiu for useful discussions and comments and to the unknown referee for helpful suggestions. The research of A. Bufetov on this project has received funding from the European Research Council (ERC) under the European Union's Horizon 2020 research and innovation programme under grant agreement No 647133 (ICHAOS). A. Bufetov has also been funded by RFBR grant 18-31-20031 and the Grant MD 5991.2016.1 of the President of the Russian Federation, by  the Russian Academic Excellence Project `5-100' and by the Chaire Gabriel Lam\'e at the Chebyshev Laboratory of the SPbSU, a joint initiative of the French Embassy in the Russian Federation and the Saint-Petersburg State
University.  The work of R. Romanov was supported by the Russian Science Foundation Grant 17-11-01064 (Theorems 3.1 and 3.5). R. Romanov gratefully acknowledges the hospitality of CNRS Institut de Math{\'e}matiques de Marseille.

\section{The weak division property and integrability}

\begin{theorem} \label{integrable}
Let $ H $ be a reproducing kernel  subspace in $ L^2 ( \R , \mu ) $, and assume  that for any $ k \in U $ and $ f \in H $ satisfying  $ f ( k ) = 0 $ there exists a unique function $ g \in H $ such that   
$ f ( x ) =  ( x - k ) g ( x )$
for all $ x \in U $. Then there exist functions $A$, $B$ defined on $U$ such that for all $x,y\in U$, $ x \ne y $, the reproducing kernel $ K $ of the space $ H $ admits the integrable representation 
\begin{equation} \label{integ}
K(x,y)=\displaystyle \frac{A(x)\overline{B(y)} - B(x)\overline{A(y)}}{x-y}.
\end{equation}
\end{theorem}

\begin{remark} The theorem holds both in the continuous and the discrete setting: for example, the measure  $\mu$ can be the counting measure on $\mathbb Z$. 
\end{remark}
\begin{remark}
A reproducing kernel subspace $H $ is called real if for any $ f \in H $ the function $ f^*( k ) \colon = \overline{ f ( k ) } $ belongs to $ H $.
If the space $ H $ is real then the functions $ A $ and $ B $ in (\ref{integ}) can be chosen real.  
\end{remark}
\begin{remark}
The term  integrable comes from the connection with the theory of integrable systems discovered by Its, Izergin, Korepin and Slavnov in \cite{IIKS}.
\end{remark}
\begin{proof} 
Fix an arbitrary point $p\in U$ such that $ K_p \not\equiv 0 $. To a point $x\in U$  assign a function $\phi_x$ by the formula 
\begin{equation}\label{def-phix}
( t-p) \phi_x(t)= K_x(t)K_p(p)-K_x(p)K_p(t) . 
\end{equation}
We have $ ( t - p ) \phi_x \in H $ for any $ x \in U $. 
By the division property, we have $ \phi_x \in H $, and the relation  
\begin{equation}\label{self-adj1}
\langle (t-p)\phi_x, \phi_y\rangle=\langle \phi_x ,  (t-p)\phi_y \rangle
\end{equation}
holds for all $ x , y \in U $.
Expanding the inner products in (\ref{self-adj1})  and using the reproducing kernel property we find 
\[ 
K(p,p)\overline{ \phi_y(x)} - K_x (p)\overline{\phi_y(p)} = K(p,p)\phi_x(y)- \overline{K_y (p)} \phi_x(p)
\] 
Substituting the definition (\ref{def-phix}) and multiplying by $(K(p,p))^{-1}(x-p)(y-p)$, we obtain
\begin{multline}\label{kpxy}
(x-p)[K(p,p)K_x ( y )  - K_x(p) K_p (y) ]- (x-p)(y-p)\overline{K_y (p)} \frac{\phi_x (p)}{K_p(p)}=\\=
(y-p)[K(p,p)\overline{K_y(x)} - \overline{ K_y (p)} \overline{K_p(x)} ]- (x-p)(y-p)K_x (p)\frac{\overline{ \phi_y (p)}}{ K_p (p ) }.
\end{multline}
Set 
\begin{equation}\label{AB}
A_p(x)=(x-p)K(x,p), B_p(x)= K(x,p)-(x-p)\frac{\overline{\phi_x(p)}}{K_p ( p ) }. 
\end{equation}

From (\ref{kpxy}) we then have 
$$
K(x,y)=\displaystyle \frac1{K(p,p)}\displaystyle \frac{A_p(x)\overline{B_p(y)} -\overline{A_p(y)} B_p(x)}{x-y} ,
$$
and the proof is complete.
\end{proof}

Note that a pair of functions $ A $ and $ B $ for which (\ref{integ}) holds is recovered explicitly in terms of the reproducing kernel by  (\ref{AB}).  This pair is non-unique, however the recovery procedure is stable in the following sense. 

\begin{proposition}
Let $ H_n $, $ n > 0 $, be a sequence of real reproducing kernel subspaces in  $ L^2 ( \R , \mu ) $ having the same set $ U $ and satisfying the condition of Theorem \ref{integrable} for each $ n $, $ K^n $ be the corresponding reproducing kernels. Assume that for each $ w \in U $ the functions $ K^n_ w $ converge  to a non-zero limit, $K_w$, in  $ L^2 ( \R , \mu ) $ as $ n \to \infty $. Then there exist sequences $ A_n $, $ B_n $ of functions on $ U $ such that 
\begin{enumerate} 
\item for any $ n > 0 $ \[ K^n (x,y)=\displaystyle \frac{A_n (x)\overline{B_n (y)} - B_n (x)\overline{A_n(y)}}{x-y}; \]
 
\item $ A_n $ and $ B_n $ belong to the space $  X = L^2 ( \R , M) $, $ dM ( x ) = \left( 1 + x^2 \right)^{ -1 } \mathrm{d}\mu ( x )  $, and converge in $ X $ as well as pointwise.
 \end{enumerate}
\end{proposition}  

\begin{proof}
Fix an arbitrary $ p \in U $. We will omit the subscript $ p $ in formulae (\ref{AB}). Note first that $ A_n $ defined to be the right-hand side of the first equality in (\ref{AB}) with $ K = K^n $, belong to $ X $ for each $ n $ and converge in $  X $. Then $ K^n ( x , y ) = \langle K^n_y , K^n_x \rangle $ converge as well for each $ x, y \in U $, hence $ A_n ( y ) $ converges for each $ y \in U $. Fix an arbitrary $ y_* \in U $ such that $ \lim A_n ( y_* ) \ne 0 $. Such a point $ y_* $ exists since otherwise $ K^n_y \to 0  $ in $  L^2 ( \R , \mu ) $ for each $ y \in U $ which contradicts the assumption. Let $ \tilde B_n $ be the right-hand side of the second equality in (\ref{AB}) with $ K = K^n $. Assume that $ N $  is large enough in such a way that
 $ A_n ( y_* ) \ne 0 $ for $ n > N $. For $ n > N $ define 
\[  B_n = \tilde B_n - \frac{ \tilde B_n ( y_* )}{ A_n ( y_* )} A_n . \]
For all $ x , y \in U $, $ x \ne y $, we then have
\[
K^n (x,y)=\displaystyle \frac1{K^n (p,p)}\displaystyle \frac{A_n (x) B_n(y) - A_n(y) B_n(x)}{x-y} .
\]
In particular, for $ y = y_* $ we obtain,
\[  B_n ( x ) = \frac{K^n ( p , p )}{ A_n ( y_* ) } K^n ( x , y_* ) ( y_* - x ) . \]
The sequence $ K^n ( p , p ) $ has finite limit, and the denominator has non-zero limit by the choice of $ y_* $. It follows that $ B_n \in X $ for all $ n > 0 $, and converges both pointwise and in $ X  $.\end{proof}  

\section{Division properties and analyticity}\label{dca}

From now on we assume that the measure $ \mu $ has no atoms. 
Throughout the rest of the paper expressions of the form $ g / ( x - p ) \in H $ stand for the unique function $ f \in H $ such that $ f ( x ) ( x - p ) = g ( x ) $. We write 
$\mathrm{clos}\,  U $ for the closure of $U$.
\begin{theorem} \label{strdiv-anal}  Let $ H $ be a reproducing kernel subspace in $ L^2 ( \R , \mu ) $ such that for any $ k \in U $ we have the implication   
\begin{equation}\label{str-cond} f \in H \Rightarrow  \frac{ f (x )  - f( k  )}{ x - k } \in H. \end{equation}
 Let $ \mathcal N = \left\{ z \in \C \setminus\partial U \colon ( \cdot - z )^{ -1 } \in H \right\} $. Then $ \mathcal N $ is a subset of $ \C \setminus \mathrm{clos}\,  U $, discrete in $ \C \setminus \partial U $, and any function from $ H $ extends to a meromorphic  function on $ \C \setminus  \partial U  $ with all poles lying in $ \mathcal N $.
\end{theorem}

The class of spaces satisfying (\ref{str-cond}) includes all regular de Branges spaces of entire functions. If $ \mathcal H ( E ) $ is a regular de Branges space \cite{deBr} corresponding to a Hermite--Biehler function $ E $, then the condition is satisfied with $ d\mu = |E|^{ -2 } dt $, $ U = \R $. The class is not reduced to regular de Branges spaces, since adding any rational fraction, $ ( t - \lambda )^{ -1 } $, $ \lambda \notin \R $, to a given space $ H $ produces another space satisfying (\ref{str-cond}). 

\begin{lemma}\label{localext}
Let $ \nu $ be a measure on $ \C $ without atoms, let $ U $ be a Borel subset such that $ \nu(\C\setminus U)=0 $. Let $ L $ be a reproducing kernel subspace of $ L^2 ( \C , \nu ) $ of functions on $ U $, such that for any $ f \in L $, $ k \in U $ we have  $$ \frac{ f (x )  - f( k  )}{ x - k } \in L. $$ Then any point $ w \in U $ has a complex neighbourhood, $ V_w $,  such that any function $ f \in L $ admits analytic continuation to $ V_w $.  The analytic continuation is given by the  formula
\begin{equation}\label{ancont} g  ( w + \lambda^{ -1 } ) = - \lambda \left \langle  ( D_w - \lambda )^{ -1 } g , K_w \right \rangle , \end{equation}
where $ D_w \colon L \to L $ is the operator defined by  the formula
\[  D_w \colon f \mapsto \frac{ f (x )  -  f (w )}{ x - w } . \] 
\end{lemma}

\begin{proof}
Take a point $ w \in U $. The closed graph theorem and the reproducing kernel property imply that $ D_w $ is a bounded operator in $ L $. We now consider the resolvent of the operator $D_w$. Solving the equation
\[ \frac{ f (x )  - f(w)  }{ x - w } - \lambda f ( x ) = g ( x )  \]
for $ f $ and using  the identity $ ( r -s ) D_{r } D_{ s } =  D_r - D_s $ valid for all $ r,s \in U $,   for any $ \lambda \in \rho ( D_w ) $ such that $ w + \lambda^{ -1 } \in U $ we find 
\begin{equation}\label{dkres} \left( ( D_w - \lambda )^{ -1 } g\right)  ( x ) =\frac{  ( x-w )g ( x ) - \lambda^{ - 1 } g ( w  + \lambda^{ -1 })}{ 1 - \lambda ( x-w ) } , \end{equation} 
that is, $$ \left( D_w - \lambda \right)^{ -1 } =  - \lambda^{ -1 } I - \lambda^{-2} D_{ w + \lambda^{ -1 } } .$$

Evaluating (\ref{dkres}) at $ x = w $ we obtain (\ref{ancont}). Since $ D_w $ is a bounded operator, the right-hand side of (\ref{ancont}) is analytic in $ \lambda $ in a neighbourhood of infinity, and tends to $  \langle  g , K_w \rangle = g ( w ) $ as $ \lambda \to \infty $. Thus the right-hand side of (\ref{ancont}) indeed defines the claimed analytic continuation. 
\end{proof}

\medskip
\noindent\textit{Proof of Theorem \ref{strdiv-anal}.}  We use the formula (\ref{ancont}) for further analytic continuation of $ g \in H $, and, to this end, we localize the essential spectrum of $ D_w $. Given a point $ w \in  U $ let $ P $ be the orthogonal projection in $ H $ on the subspace of functions vanishing at $ w $. Then $ A=  P \left( x - w \right)^{ -1 } P $ is a bounded self-adjoint operator in $ H $. We clearly have  $ Af = P D_w f $ whenever $ f \in H $, $ f ( w ) = 0 $, hence $ A $ is a rank 2 perturbation of $ D_w $.  Note that \begin{equation}\label{sigmaa} \sigma_{ess} ( A ) \subset \mathrm{clos}\left({ ( \mathrm{supp } \, \mu - w )^{ -1 } } \right)\end{equation} Indeed, if $ x_0 \in \sigma_{ess} ( A ) $ then there exists an orthonormal sequence, $ e_n $, such that $ ( A - x_0 ) e_n \to 0 $. Without loss of generality, one can assume that $ e_n \in \operatorname{Ran} P $, so that 
\[ P \frac 1{x-w} e_n - x_0 e_n \to 0 . \]
We have
\[ P \frac 1{x-w} e_n = \frac 1{x-w} e_n - ( I_H - P ) D_w e_n . \]
The second term in the right-hand side goes to $0$ in the strong sense as $ n \to \infty $ because $ D_w $ is bounded, $ I_H -P $ is rank $1$, and $ e_n $ goes to zero weakly. Thus 
\[ \frac 1{x-w} e_n - x_0 e_n \to 0 . \] 
It follows that $ x_0 $ belongs to the essential spectrum of the unbounded operator of multiplication by $ 1/ ( x-w ) $ in $ L^2 ( \R , \mu ) $, that is, $ x_0 \in \mathrm{clos}\left({ ( \mathrm{supp } \, \mu - w )^{ -1 } } \right)$, and \eqref{sigmaa} is proved.

By the theorem on preservation of the essential spectrum under relative compact perturbations \cite[Theorem IV.5.35]{Kato} we find that 
\begin{equation}\label{inclus} \sigma_{ess} ( D_w )  \subset \mathrm{clos}\left({( \mathrm{supp }\,  \mu - w )^{ -1 } }\right) . \end{equation}
For completeness, let us sketch the proof of this fact. Since $ D_w $ is a finite rank perturbation of $ A $, the sets $ \{ \lambda \in \C \colon ( A - \lambda ) \textrm{ is Fredholm} \} $ and $ \{ \lambda \in \C \colon  ( D_w  - \lambda ) \textrm{ is Fredholm} \} $ coincide. As $ A $ is a bounded self-adjoint operator, the former set is connected, contains a neighbourhood of infinity, and its complement is $ \sigma_{ ess } ( A ) $. Since $ D_w $ is bounded, $ D_w - \lambda   $ has bounded inverse defined on the whole of $ H $ for $ |\lambda | $ large enough, and therefore for all $ \lambda \notin \sigma_{ess} ( A ) $ save for a discrete subset in $ \C \setminus \sigma_{ ess }( A )  $ whose points are eigenvalues of $ D_w $, by the preservation of the Fredholm index on the connectivity component.  It also follows that the invariant subspaces corresponding to these eigenvalues are finite-dimensional, hence $ \sigma_{ ess} ( D_w ) = \sigma_{ess} ( A ) $. 
  
It follows from (\ref{dkres}) that any point $ \lambda $ satisfying $ w + \lambda^{ -1 } \in U $ belongs to the resolvent set of $ D_w $, that is, setting 
 $\hat\rho_w \equiv \hat{\C}\setminus \sigma ( D_w )$, we have 
 $$ U \subset {\hat\rho}_w^{ -1 } + w. $$ This and (\ref{inclus}) combined imply that the complement of the set $  {\hat\rho}_w^{ -1 } + w $ is contained in the union of $ \partial U $ and a discrete subset, $ M $, of $ \C \setminus \partial U $ which does not intersect $ U $. In particular, the set $  \hat{\rho}_w^{ -1 } + w $ has one connected component, and it follows that any function $ g \in H $ admits meromorphic continuation to the set $ \C \setminus \partial U $ with all poles lying in $ M $. By construction for any $ \lambda \in M $  the number $(  \lambda - w )^{ -1 } $ is an eigenvalue of $ D_w $ with the eigenfunction $ ( \cdot - \lambda )^{ -1 } $. It follows that  $ M = \mathcal N $, and we arrive at the required assertion. \hfill $ \Box $

The problem of description of subspaces satisfying the assumption of Theorem \ref{strdiv-anal} will be considered elsewhere; here we limit ourselves to the following  elementary observation.

\begin{remark} 
The set $ \mathcal N $ in Theorem \ref{strdiv-anal} satisfies the Blaschke condition in the sense that $ \pm \sum_{ \mathcal N \cap \mathbb C_\pm } \Im z / ( 1 + | z |^2 ) < \infty $.
\end{remark}
\begin{proof}
Indeed, assume that the Blaschke condition is violated in, say, $ \C_- $. A standard argument shows that then the space $ H $ contains the space $ H^2_+ ( \mu ) $, the closure of the rational functions with poles in $ \C_- $ in the norm of $ L^2 ( \R, \mu ) $. There exists a set, $ \Omega \subset \R  $, of $ \mu $-measure zero such that for any  rational function, $ f $, of the form $ ( t - z )^{ -1 } $ where $ z $ ranges over rational points in $ \C_- $, the value $ \langle f , K_w \rangle $ coincides with $ f ( w ) $ for $ w \in \R \setminus \Omega $. Fix a $ y \in \R \setminus \Omega $ and an arbitrary sequence $ z_n \in y+ i \R_- $ of rational points converging to $ y $. Then 
\[  f_n ( t ) = ( y - z_n ) \left( t - z_n  \right)^{ -1 } \left( t + i \right)^{ -1 } \] 
converges to zero in $ L^2 (\R,  \mu ) $ as $ n \to \infty $, hence $ f_n ( y) = \langle f_n , K_y \rangle \to 0 $, a contradiction.  
   \end{proof}   

 \textbf{Example.} Let $ \mu $ be the Lebesgue measure on $ \R $, $ U = \R $. In this case one can describe all the real spaces $ H $ obeying the property (\ref{str-cond}). According to Theorem \ref{strdiv-anal} the elements of the space $ H $ extend to meromorphic functions with poles in a set $ \mathcal N \subset \C \setminus \R $. We will assume for brevity that the eigenspaces of $D_w $ corresponding to the points of $ \mathcal N $ are one-dimensional. Then the space $ H $ can be represented as a linear sum, $ H = \mathrm{clos}({X \dot{+} Y }) $, of the space $$ X = \bigvee_{ \lambda \in \mathcal N } \frac 1{ x - \lambda } $$ and  a space of entire functions $ Y $. The space $ Y $ is a de Branges subspace in $ L^2 ( \R ) $. Condition (\ref{str-cond}) means that $ Y $ is a regular de Branges space, and therefore it coincides with the Paley--Wiener space $ PW_a $ for some $ a > 0 $. Let $ \Theta_\pm  $ be the Blaschke products corresponding to the sets $ \mathcal N \cap \C_\pm $, respectively. By the elementary theory of the shift operator \cite{Nik-easy} the space $ X $ then coincides with $ K_+ \oplus K_- $, $ K_\pm \colon = H^2_\pm \ominus \Theta_\pm H^2_\pm $. Since the inner functions $ e^{ iaz} $ and $ \Theta_\pm $ are relative prime, we find that 
 \begin{equation}\label{H} H = \left( \Theta_+ e^{ ika} H^2_+ \oplus \Theta_- e^{ -ika} H^2_- \right)^\perp . \end{equation}  
We have established that any subspace in $ L^2 ( \R ) $ with the division property (\ref{str-cond}) has the form (\ref{H}). Conversely, for any set $ \mathcal N \subset \C \setminus \R $ satisfying the Blaschke condition in each halfplane $ \C_\pm $ and having no accumulation points at finite distance the space $ H $ defined by (\ref{H}) obeys (\ref{str-cond}). In particular, the sine process corresponds to the case when $ \mathcal N $ is empty, and $ H = PW_a $.

The Hitt--Sarason theorem \cite{Sarason} gives
a complete description of the subspaces $ X $ of the Hardy class $ H^2 $ in the unit disc such that $ f \in X $, $ f(0) = 0 $ implies $ f( z ) /z \in X $ and allows one to generalize the example just considered.
 
\begin{remark}  In the  standard examples of determinantal processes the set $ \mathcal N $ is empty. It would be interesting to give a probabilistic interpretation of processes with non-empty $ \mathcal N $. This question makes sense already in case when $ \mathcal N $ consists of a single point, $ \lambda \in \C_- $, and $ H = PW_a \dot{+} \mathcal L \{ ( t-\lambda )^{ -1 } \}$, so that the corresponding reproducing kernel has the form
 \[ K ( x , y ) =  \frac{ \sin ( x - y )}{ \pi ( x -y )}  + C \frac{ e^{ i ( x -y) } }{  (y - \overline \lambda )( x - \lambda ) } \]
 where  $ C $ is a normalization constant depending on $ \lambda $. 
\end{remark} 

According to the following theorem, the weak division property (\ref{weakdiv}) implies an analyticity property for ratios of functions from $ H $, rather than the functions themselves. To illustrate this point by a simple example,  let $ \rho $ be a positive function on $ \R $, not analytic at any point, and let $ H = \rho PW_a $. Then $ H $ is a subspace in $ L^2 ( \R , \rho^{ -2 } dt ) $ which satisfies (\ref{weakdiv}) but no non-zero element of $H$ admits analytic continuation. 

Unlike the strong division property (\ref{str-cond}), condition (\ref{weakdiv}) allows for the situation when all the functions in $ H $ vanish on a set $ \Omega \subset U $.  Should this happen, one can always consider the measure $ \chi_{ \R \setminus \Omega } d \mu $ and the set $ U \setminus \Omega $ instead of $ \mu $ and $ U $ respectively. Without loss of generality we can therefore assume that the space $ H $ separates zeros, that is, for any $ p \in U $ there exists an $ f_p \in H $ such that $ f_p ( p ) \ne 0 $. 

\begin{theorem}\label{weakdiv-anal} 
Let $ H $ be a reproducing kernel  subspace in $ L^2 ( \R , \mu ) $ that separates zeros 
and such that for any $ k \in U $ we have the implication   
\[ f \in H , \; f ( k ) = 0 \Longrightarrow  \frac{ f }{ x - k } \in H .\]
 Then for any $ p \in U $ the following holds.
\begin{enumerate}
\item The set of accumulation points of zeros of  the function $K_p$  is contained in $\partial U \cup \{ \infty \}$. 
   \item  There exists  an open neighbourhood $ V_p  \subset \C $ of $ p $ such that for any $ f \in H $ the  function $ f/K_p $ extends to an analytic function in $ V_p $.
\item  There exists 
a discrete subset $ \mathcal N_p\subset \R \setminus  \partial U  $, such that for any $ f \in H $ the function $ f/K_p $ extends to a meromorphic function in $ \C \setminus \partial U $  all whose poles are contained in $ \mathcal N_p $.
\end{enumerate}
\end{theorem}

\begin{proof} Let $ p \in U $ and set $ \rho ( p ) = 1/K_p ( p ) $. Define the operator, $ D_p \colon H \to H $, by the formula
\[ D_p \colon f \mapsto \frac{ f  -  \rho ( p ) f(p )K_p }{ x - p } . \]
This operator is bounded by the closed graph theorem (we have already used this argument  in the proof of Theorem \ref{strdiv-anal}). Solving 
\[ \frac{ f(x)  - \rho ( p ) f(p )K_p(x) }{ x - p } - \lambda f ( x ) = g ( x )  \]
for $ f $,  we find that for any $ \lambda \in \rho ( D_p ) $ there exists a constant $ C ( p , \lambda ) $ such that
\begin{equation}\label{resolvent} \left( ( D_p - \lambda )^{ -1 } g\right)  ( x ) = \frac{  ( x-p ) g ( x ) + C( p , \lambda ) \rho ( p ) K_p ( x ) }{ 1 - \lambda ( x-p ) }  . \end{equation}
Assume $ p + \lambda^{ -1 } \in U $. Substituting  $ x = p + \lambda^{ -1 } $ we have 
\[ \frac 1\lambda g \left( p + \lambda^{ -1 } \right) + C ( \lambda , p ) \rho ( p ) K_p \left( p + \lambda^{ -1 } \right) = 0 . \]
Note that $ K_p ( p + \lambda^{ -1 } ) \ne 0 $, for  otherwise we would have $ g ( p + \lambda^{ -1 } ) = 0 $ for all $ g \in H $ in contradiction with the assumption that $ H $ separates zeros. Consequently, 
\begin{equation}\label{c} C ( \lambda , p ) = - \lambda^{ - 1 }   K_p ( p )  \frac{g( p  + \lambda^{ -1 })} { K_p \left( p + \lambda^{ -1 } \right) } . \end{equation}
The same calculation shows that if $ K_p ( p + \lambda^{ -1 })  \ne 0 $ then $ \lambda \in \rho ( D_p ) $ and the resolvent is given by the formulae (\ref{resolvent}) and (\ref{c}). Thus a point $ \lambda $ satisfying $ p + \lambda^{ -1 } \in U $ belongs to the resolvent set of  $ D_p $ if and only if $ K_p ( p + \lambda^{ -1 } ) \ne 0 $. Since $ D_p $ is a bounded operator we infer that for a sufficiently small $ \epsilon_p > 0 $ we have  $ K_p ( t) \ne 0 $ for all $ t \in U \cap ( p - \epsilon_p , p + \epsilon_p)$.   

Evaluating (\ref{resolvent}) at $ x = p $ we obtain  that for any $ \lambda $ such that  $ p + \lambda^{ -1 } \in U $ and $ K_p ( p + \lambda^{ -1 } ) \ne 0 $ we have 
\begin{equation}\label{reso-analy} \frac{g( p  + \lambda^{ -1 })} { K_p \left( p + \lambda^{ -1 } \right) } =  - \lambda \rho ( p ) \left \langle  ( D_p - \lambda )^{ -1 } g , K_p \right \rangle  \end{equation}
This formula defines an analytic continuation of the function $ g / K_p $ to an open neighbourhood, $ V_p $, of $ p $ with $ V_p $ independent of $ g $. In particular, the set of zeroes of $ g $ lying on $ U \cap ( p - \epsilon_p , p + \epsilon_p )$ is discrete for $ \epsilon_ p > 0 $ small enough unless $ g \equiv 0 $. Varying $ p $ over $ U $ we obtain that for each $ p \in U $ there exists an interval $ I_p $, $ p \in I_p $, such that $ g $ vanishes on $ U \cap I_p $ at an at most discrete set. It follows that the zeros of $ g $ can only accumulate at $ \partial U \cup \infty $. Applying it to $ g = K_p $ and taking into account that if $ K_p ( p + \lambda^{ -1 })  \ne 0 $ then $ \lambda \in \rho ( D_p ) $, we find that $ \rho ( D_w )^{ -1 } + w $ contains all points in $ U \setminus \Sigma $ for a discrete subset $ \Sigma \subset U $. Arguing as in Theorem \ref{strdiv-anal}, we see that  $ \C \setminus \mathrm{supp} \,\mu $ is contained in $ \rho ( D_w )^{ -1 } + w  $ except possibly  for a subset discrete in $ \C \setminus \mathrm{supp} \,\mu $. It follows now from (\ref{reso-analy}) that the function $ f /K_p $ extends to a function meromorphic in the complement of the set 
$ \Omega = \{ z \in \R \colon K_p (z ) = 0 \}  \cup \partial U $ with poles in a set $ \mathcal N_p $ discrete in $ \C \setminus  \Omega $ and independent of $ f \in H $.  In fact, $ \mathcal N_p $ is a subset of $ \sigma_d ( D_p )^{ -1 } + p $. 

It remains to check that $ \mathcal N_p \subset \R $ and that the points $ z \in \Omega $, $ z \notin \partial U $, are actually either poles, or regular, for $ f / K_p $.  Let  $ Q $ be the orthogonal projection in $ H $ on the subspace of functions vanishing at $ p $, $ A $ be the operator of multiplication by the function $ \left( p - \cdot \right)^{ -1  } $ defined on $ \operatorname{Ran} Q $ by the assumption. With this notation, $ D_p = A Q $. Define a bounded operator, $ T_p $, in $ H $ by $ T_p = Q A Q $. Obviously, $ T_p $ is a rank 1 perturbation of $ D_p $. We claim that $ \sigma_p (T_p) = \sigma_p ( D_p ) $. Indeed the obvious relation $ Q D_p = T_p Q $ implies that $ \sigma_p ( D_p  ) \subset \sigma_p ( T_p ) $.  In the opposite direction, let $ f \in H $ satisfy $ Q A Q f = \lambda f $, $ \lambda \ne 0 $, then $ f \in \operatorname{Ran Q } $ hence $ Q A f = \lambda f $. Applying $ A $ to this equality we find that $ D_p g = \lambda g $ for $ g = Af $. The inclusion $ \mathcal N_p \subset \R $ is now immediate since the operator $ T $ is self-adjoint.

Let $ z $ be an interior point of $ U $ such that $ K_p ( z ) = 0 $, and let $ w $ be defined by $  z = p + w^{ -1 } $. Then $ z $ is not an accumulation point of poles of $ f/K_p $, since the zeroes of $ K_p $ are discrete in $ U $, and therefore a punctured neighbourhood of $ w $ belongs to $ \rho ( D_w ) $. It follows that $ w $ is an isolated point of $ \sigma ( D_p ) $. 
Using again the theorem on preservation of the essential spectrum under relative compact perturbations \cite[Theorem IV.5.35]{Kato}, this time in another direction, we find that
 $ w $ is either an accumulation point for eigenvalues, or an isolated eigenvalue,  for the operator $ T_p  $. The former is ruled out by  the equality $ \sigma_p (T_p) = \sigma_p ( D_p ) $. A straightforward calculation shows that a non-zero eigenvalue of $ T_p $ is simple, and the eigenvalue $ 0 $ has multiplicity at most $ 2 $. It follows that $ w $ belongs to the discrete spectrum of $ T_p $, and therefore, of $ D_p $. This implies that the resolvent of $ D_p $ has a pole at $ w $, from whence the point $ z $ is not an essential singularity of $ f / K_p $, as required.  
\end{proof}

\begin{remark} The setting of Theorem \ref{weakdiv-anal}, in which the ratios of elements of $ H $ are analytic, but the elements themselves are not,  is realized, in particular,   for the Bessel kernel where the functions $f \in H $ are only analytic in the slit plane $ \C \setminus \{ t \colon t \le 0 \} $. 
\end{remark} 

If the set $ U $ is open, then we say  that the reproducing kernel is locally of trace class if for any compact $ \Omega \subset U $ the operator $ \chi_\Omega K \chi_\Omega $ is of  trace class; here $ \chi_\Omega $ stands for the operator of multiplication by the indicator function of the set $ \Omega $, and $ K $ is the integral operator in $ L^2 ( \R , \mu ) $ defined by the kernel $ K ( x , y ) $. 

\begin{corollary}\label{loctraceclass}
Under the  assumptions of  Theorem \ref{weakdiv-anal} the reproducing kernel is  locally of trace class.
\end{corollary}

\begin{proof}
Let $ A $ and $ B $ be the functions from the representation of the reproducing kernel of Theorem \ref{integrable}. By Theorem \ref{weakdiv-anal}, for each $ p \in U $ there exists a  neighbourhood, $ V_p \subset \C $, of $ p $ such that the ratio $ K_x / K_p $ is analytic in it  for all $ x \in U $, that is the function 
\[ F ( w ) = \frac{ A ( x ) B ( w ) - B ( x ) A ( w) }{ A ( p ) B ( w ) - A ( w ) B ( p ) } \frac{ p-w}{ x-w} \] 
is analytic in $ w $ in $ V_p $. Let 
\[ m ( w )  = \frac{ A ( w )}{ B ( w ) } . \] 
Assume first that $ B ( p ) = 0 $.  Then $ A ( p ) \ne 0 $, for otherwise $ K_p \equiv 0 $, and we infer that the function 
\[  \left( A ( x )  - B ( x ) m ( w )\right) \frac{ p-w}{ x-w} \]
is analytic in $ V_p $. Picking an $ x $ such that $ B ( x ) \ne 0 $ (this is possible for otherwise $ B \equiv 0 $, whence $ K_x \equiv 0 $) we obtain that the function $m $ admits analytic continuation to $ V_p $ except for a possible simple pole at $ p $. If $ B ( p ) \ne 0 $ then the function $ m $ extends to a meromorphic function on $ V_p $ with poles contained in the set 
\[ \left\{ w \colon \frac{ x-w}{ p-w}F ( w ) = \frac{ B ( x ) }{ B ( p ) }  \right\} , \] 
which can only accumulate at the boundary of $ V_p $. Thus the function $ m $ admits a meromorphic continuation to $ V_p $ for every $ p \in U $.

The operator $ K $ is an orthogonal projection in $ L^2 ( \R , \mu ) $, hence $ \chi_\Omega K \chi_\Omega = \chi_\Omega K \cdot K \chi_\Omega $, and the required assertion will be established if we show that $ \chi_\Omega K $ is a Hilbert--Schmidt operator, that is,
\[  \int_{\Omega \times \R } | K ( x , y ) |^2 \mathrm{d} \mu ( x ) \mathrm{d} \mu ( y ) < \infty \] 
for any compact $ \Omega \subset \R $. In doing so, we will assume that $ A $ and $ B $ are chosen in such a way that each of them vanishes at least at one point. This can always be achieved by a linear transformation.

Let us first consider the integral away from the diagonal. 
 For any $ \epsilon > 0 $ we have
\begin{equation}\label{main-ineq}
\displaystyle \int\limits_\Omega  \mathrm{d} \mu ( x )  \displaystyle \int\limits_{ | x - y | > \epsilon } | K ( x , y ) |^2 \mathrm{d} \mu ( y ) \le  C_\epsilon \displaystyle \int\limits_ \Omega A^2 ( x ) \mathrm{d} \mu ( x ) \displaystyle \int\limits_\R B^2 ( y )\frac { \mathrm{d} \mu ( y ) }{ 1+ y ^2  } . 
\end{equation}
Let $p$ and $s$ be such that $A(p) =0$, $B(s)=0$. Both integrals in the right-hand side of (\ref{main-ineq})   are finite because the functions $ \displaystyle \frac{A ( t )} {t-p }  \ \mathrm{and} \    \displaystyle\frac{ B ( t )}{ t-s}   $ are essentially the reproducing kernels at points $ p $ and $ s$ respectively, and thus belong to $ L^2 ( \R ,  \mu ) $. 

It remains to check that for some $ \epsilon > 0 $ the kernel $ K ( x , y ) $ is square-integrable over the set $$ \{ ( x , y ) \colon x \in \Omega , y \in U, | x -y | \le \epsilon \} .$$
Let $ r_p $ be the radius of the neighbourhood $ V_p $, and set $$ I_p = ( p - r_p/4 , p+r_p/4 ) , \  J_p = ( p - r_p/2 , p + r_p/ 2).$$
 Since $ \Omega $ is compact, it suffices to check that $ K $ is square integrable over the set $$ \Pi = ( I_p \cap U ) \times ( J_p \cap U ).$$ 
  Note that 
\begin{equation}\label{diagonal} K ( x , y ) =  B ( x ) B ( y ) \frac{ m ( x ) - m ( y ) }{ x -y } . \end{equation}
Let $ G $ and $ G^\prime $ be the complements of  sufficiently small neighbourhoods of poles and zeros of $ m $ in $ V_p $, respectively. Since $ B \in L^2_{ loc} ( \R , \mu ) $, we have
\begin{equation} \label{intG} \int_{ ( G \times G) \cap \Pi } | K ( x , y )  |^2 \mathrm{d} \mu ( x ) \mathrm{d} \mu ( y )  < \infty . \end{equation}
Rewriting (\ref{diagonal}) in the form 
\[ K ( x , y ) =  A ( x ) A ( y ) \frac{ m^{ -1 } ( y ) - m^{ -1 } ( x ) }{ x -y } \]
by the same argument we find that the integral in (\ref{intG}) is finite with $ G $ replaced by $ G^\prime $. Since $ m $ is meromorphic in $ V_p $, there are at most finitely many zeros and poles of $ m $ in $ J_p$. Taking $ \epsilon $ less than half the distance between the sets of zeros and poles, and choosing  the neighbourhoods small enough, one achieves the inclusion 
$$ \Pi \cap \{(x,y): |x-y|<\varepsilon \}\subset (G \times G) \cup (G^\prime  \times G^\prime),$$
 and the proof is complete.
\end{proof}

This corollary implies in particular that point (2) in Assumption 1 in \cite{Buf-gibbs} can be omitted in the continuous case, that is, it  follows from the other assumptions of the main Theorem 1.4 in that paper.

\end{document}